\documentclass[twoside]{amsart}

\usepackage{amsmath,amsfonts,amsthm,mathrsfs}
\usepackage{amssymb}

\usepackage[unicode,bookmarks]{hyperref} 
\usepackage[usenames,dvipsnames]{xcolor}
\hypersetup{colorlinks=false}

\usepackage{expl3}

\ExplSyntaxOn
\prop_new:N \g_cite_map_prop
\tl_new:N \l_citekey_result_tl

\cs_new:Npn \mapcitekey #1#2 {
  \clist_map_inline:nn {#2}
       {  \prop_gput:Nnn  \g_cite_map_prop  {##1} {#1}   }
}

\cs_new:Npn \getcitekey #1 {
   \prop_get:NoN \g_cite_map_prop{#1}  \l_citekey_result_tl
   \quark_if_no_value:NF \l_citekey_result_tl
       {  \tl_set_eq:NN #1  \l_citekey_result_tl  }
}

\cs_new:Npn \showcitekeymaps {\prop_show:N  \g_cite_map_prop }
\ExplSyntaxOff

\usepackage{etoolbox}
\makeatletter
\patchcmd{\@citex}{\if@filesw}{\getcitekey\@citeb \if@filesw}%
    {\typeout{*** SUCCESS ***}}{\typeout{*** FAIL ***}}
\patchcmd{\nocite}{\if@filesw}{\getcitekey\@citeb \if@filesw}%
    {\typeout{*** SUCCESS ***}}{\typeout{*** FAIL ***}}
\makeatother

\usepackage{enumitem}

\usepackage{chngcntr}

\usepackage{tikz}
\usepackage{tikz-cd}
\tikzset{commutative diagrams/arrow style=Latin Modern}

\newcommand{\shH}{\mathcal{H}}
\newcommand{\shA}{\mathcal{A}}
\newcommand{\Mmod}{\mathcal{M}}
\newcommand{\Dmod}{\mathscr{D}}

\newcommand{\abs}[1]{\lvert #1 \rvert}

\newcommand{\tensor}{\otimes}

\newcommand{\ZZ}{\mathbb{Z}}
\newcommand{\QQ}{\mathbb{Q}}
\newcommand{\RR}{\mathbb{R}}

\newcommand{\PP}{\mathbb{P}}

\newcommand{\menge}[2]{\bigl\{ \thinspace #1 \thinspace\thinspace \big\vert%
\thinspace\thinspace #2 \thinspace \bigr\}}

\DeclareMathOperator{\rk}{rk}

\DeclareMathOperator{\id}{id}

\DeclareMathOperator{\Sym}{Sym}
\DeclareMathOperator{\gr}{gr}

\DeclareMathOperator{\Var}{Var}

\newcommand{\shf}[1]{\mathscr{#1}}
\newcommand{\OX}{\shf{O}_X}

\def\overbar#1#2#3{{%
	\setbox0=\hbox{\displaystyle{#1}}%
	\dimen0=\wd0
	\advance\dimen0 by -#2 
	\vbox {\nointerlineskip \moveright #3 \vbox{\hrule height 0.3pt width \dimen0}%
		\nointerlineskip \vskip 1.5pt \box0}%
}}

\newcommand{\into}{\hookrightarrow}

\newcommand{\fu}{f^{\ast}}

\newcommand{\tl}{t_{\ast}}

\newcommand{\shF}{\shf{F}}
\newcommand{\shG}{\shf{G}}
\newcommand{\shE}{\shf{E}}

\newcommand{\shO}{\shf{O}}

\makeatletter
\let\@@seccntformat\@seccntformat
\renewcommand*{\@seccntformat}[1]{%
  \expandafter\ifx\csname @seccntformat@#1\endcsname\relax
    \expandafter\@@seccntformat
  \else
    \expandafter
      \csname @seccntformat@#1\expandafter\endcsname
  \fi
    {#1}%
}
\newcommand*{\@seccntformat@subsection}[1]{%
  \textbf{\csname the#1\endcsname.}
}
\makeatother

\makeatletter
\let\@paragraph\paragraph
\renewcommand*{\paragraph}[1]{%
	\vspace{0.3\baselineskip}%
	\@paragraph{\textit{#1}}%
}
\makeatother

\counterwithin{equation}{subsection}
\counterwithout{subsection}{section}
\counterwithin{figure}{subsection}

\newtheorem{theorem}[equation]{Theorem}
\newtheorem*{theorem*}{Theorem}
\newtheorem{lemma}[equation]{Lemma}
\newtheorem*{lemma*}{Lemma}

\newtheorem*{corollary*}{Corollary}

\newtheorem*{proposition*}{Proposition}

\newtheorem*{conjecture*}{Conjecture}

\theoremstyle{definition}

\newtheorem*{definition*}{Definition}
\theoremstyle{remark}
\newtheorem{remark}{Remark}
\newtheorem*{remark*}{Remark}

\newtheorem*{example*}{Example}
\newtheorem*{problem*}{Problem}
\newtheorem*{note}{Note}

\theoremstyle{plain}

\newcommand{\theoremref}[1]{\hyperref[#1]{Theorem~\ref*{#1}}}
\newcommand{\lemmaref}[1]{\hyperref[#1]{Lemma~\ref*{#1}}}
\newcommand{\remarkref}[1]{\hyperref[#1]{Remark~\ref*{#1}}}
\newcommand{\definitionref}[1]{\hyperref[#1]{Definition~\ref*{#1}}}
\newcommand{\propositionref}[1]{\hyperref[#1]{Proposition~\ref*{#1}}}
\newcommand{\conjectureref}[1]{\hyperref[#1]{Conjecture~\ref*{#1}}}
\newcommand{\corollaryref}[1]{\hyperref[#1]{Corollary~\ref*{#1}}}
\newcommand{\exampleref}[1]{\hyperref[#1]{Example~\ref*{#1}}}
\newcommand{\exerciseref}[1]{\hyperref[#1]{Exercise~\ref*{#1}}}

\makeatletter
\let\old@caption\caption
\renewcommand*{\caption}[1]{%
	\setcounter{figure}{\value{equation}}%
	\stepcounter{equation}%
	\old@caption{#1}\relax%
}
\makeatother

\newcounter{intro}

\newtheorem{intro-conjecture}[intro]{Conjecture}
\newtheorem{intro-corollary}[intro]{Corollary}
\newtheorem{intro-theorem}[intro]{Theorem}
\newtheorem{intro-proposition}[intro]{Proposition}

\newcommand{\omY}{\omega_Y}

\newcommand{\OmY}{\Omega_Y}

\newcommand{\parref}[1]{\hyperref[#1]{\S\ref*{#1}}}
\newcommand{\chapref}[1]{\hyperref[#1]{Chapter~\ref*{#1}}}

\makeatletter
\newcommand*\if@single[3]{%
  \setbox0\hbox{${\mathaccent"0362{#1}}^H$}%
  \setbox2\hbox{${\mathaccent"0362{\kern0pt#1}}^H$}%
  \ifdim\ht0=\ht2 #3\else #2\fi
  }
\newcommand*\rel@kern[1]{\kern#1\dimexpr\macc@kerna}
\newcommand*\widebar[1]{\@ifnextchar^{{\wide@bar{#1}{0}}}{\wide@bar{#1}{1}}}
\newcommand*\wide@bar[2]{\if@single{#1}{\wide@bar@{#1}{#2}{1}}{\wide@bar@{#1}{#2}{2}}}
\newcommand*\wide@bar@[3]{%
  \begingroup
  \def\mathaccent##1##2{%
    \if#32 \let\macc@nucleus\first@char \fi
    \setbox\z@\hbox{$\macc@style{\macc@nucleus}_{}$}%
    \setbox\tw@\hbox{$\macc@style{\macc@nucleus}{}_{}$}%
    \dimen@\wd\tw@
    \advance\dimen@-\wd\z@
    \divide\dimen@ 3
    \@tempdima\wd\tw@
    \advance\@tempdima-\scriptspace
    \divide\@tempdima 10
    \advance\dimen@-\@tempdima
    \ifdim\dimen@>\z@ \dimen@0pt\fi
    \rel@kern{0.6}\kern-\dimen@
    \if#31
      \overline{\rel@kern{-0.6}\kern\dimen@\macc@nucleus\rel@kern{0.4}\kern\dimen@}%
      \advance\dimen@0.4\dimexpr\macc@kerna
      \let\final@kern#2%
      \ifdim\dimen@<\z@ \let\final@kern1\fi
      \if\final@kern1 \kern-\dimen@\fi
    \else
      \overline{\rel@kern{-0.6}\kern\dimen@#1}%
    \fi
  }%
  \macc@depth\@ne
  \let\math@bgroup\@empty \let\math@egroup\macc@set@skewchar
  \mathsurround\z@ \frozen@everymath{\mathgroup\macc@group\relax}%
  \macc@set@skewchar\relax
  \let\mathaccentV\macc@nested@a
  \if#31
    \macc@nested@a\relax111{#1}%
  \else
    \def\gobble@till@marker##1\endmarker{}%
    \futurelet\first@char\gobble@till@marker#1\endmarker
    \ifcat\noexpand\first@char A\else
      \def\first@char{}%
    \fi
    \macc@nested@a\relax111{\first@char}%
  \fi
  \endgroup
}
\makeatother

\mapcitekey{Popa+Schnell:OneForms}{PS1}
\mapcitekey{Popa+Schnell:Hyperbolicity}{PS2}
\mapcitekey{Viehweg+Zuo:Isotriviality}{VZ1}
\mapcitekey{Viehweg+Zuo:BaseSpaces}{VZ2}
\mapcitekey{Viehweg+Zuo:Brody}{VZ3}
\mapcitekey{Mori:Survey}{Mori}
\mapcitekey{Viehweg:WP1}{Viehweg1}
\mapcitekey{Popa:Conjectures-Kodaira}{Popa}
\mapcitekey{Meng+Popa:Kodaira-Abelian}{MP}
\mapcitekey{Kawamata:AdditivityMMP}{Kawamata}
\mapcitekey{Campana+Peternell:GeometricStability}{CP}
\mapcitekey{Campana+Paun:Quotients}{CPa}
\mapcitekey{Schnell:CP}{Schnell}
\mapcitekey{Schnell:campana-peternell}{Schnell2}
\mapcitekey{Taji:special}{Taji}
\mapcitekey{BDPP:Pseudoeffective}{BDPP}
\mapcitekey{Kollar:Subadditivity}{Kollar}
\mapcitekey{Cao:KodairaSurfaces}{Cao}
\mapcitekey{Viehweg:WP1}{Viehweg1}
\mapcitekey{Cao+Paun:Kodaira}{CaP}
\mapcitekey{Paun+Takayama:DirectImages}{PT}
\mapcitekey{Popa+Wu:WeakPositivity}{PW}
\mapcitekey{Zhang:NefAnticanonical}{Zhang}
\mapcitekey{Birkar+Chen:Images}{BC}
\mapcitekey{BCHM:FiniteGeneration}{BCHM}
\mapcitekey{Siu:Plurigenera}{Siu}
\mapcitekey{Viehweg:Moduli}{Viehweg2}
\mapcitekey{Kollar:Book}{Kollar2}
\mapcitekey{Viehweg:CurveFibrations}{Viehweg3}
\mapcitekey{Kollar+Larsen:Quotients}{KL}
\mapcitekey{Pieloch:Sections}{Pieloch}
\mapcitekey{Tosatti+Zhang:Collapsing}{TZ}
\mapcitekey{Wei+Wu:KK-Conjecture}{WW}
\mapcitekey{Kovacs+Patakfalvi:ProjectivityModuli}{KP}
\mapcitekey{Iitaka:Logarithmic}{Iitaka}
\mapcitekey{Campana:Additivity}{Campana}

\AtBeginDocument{%
   \def\MR#1{}
}

\usepackage[all,cmtip]{xy}

\renewcommand{\shG}{\mathscr{G}}
\renewcommand{\shF}{\mathscr{F}}
\newcommand{\shQ}{\mathscr{Q}}

\begin{document}

\title[On the behavior of Kodaira dimension]%
	{On the behavior of Kodaira dimension under smooth morphisms}
	
\author[M.~Popa]{Mihnea Popa}
\address{Department of Mathematics, Harvard University,
1 Oxford St., Cambridge, MA 02138} 
\email{mpopa@math.harvard.edu}

\author[Ch.~Schnell]{Christian Schnell}
\address{Department of Mathematics, Stony Brook University, Stony Brook, NY 11794-3651}
\email{christian.schnell@stonybrook.edu}



\begin{abstract} 
We prove several results on the additivity of Kodaira dimension under smooth
morphisms of smooth projective varieties. 
\end{abstract}
\maketitle

\section{Introduction}

The purpose of this paper is to prove some of the conjectures in \cite{Popa} about the
additivity of the Kodaira dimension under smooth morphisms between smooth projective
varieties. Our results are unconditional when varieties of general type are involved
in some way, or when the base space is either of dimension $\leq 3$ or a good minimal
model of Kodaira dimension zero; for other varieties, we need a version of the
non-vanishing conjecture due to Campana and Peternell \cite{CP}.
We use the convention that an \emph{algebraic fiber space} is a surjective morphism
$f \colon X \to Y$ with connected fibers between two smooth projective varieties $X$
and $Y$ (over the field of complex numbers). We usually denote the general fiber by
the letter $F$. We say that an algebraic fiber space is \emph{smooth} if $f$ is a
smooth morphism.

A particular case of Conjectures 2.1 or 3.1 in \cite{Popa} predicts that if $f\colon
X \to Y$ is a smooth algebraic fiber space, and $X$ is of general type, then $Y$ is of
general type as well. Our first result is the following more general statement:

\begin{intro-theorem}\label{thm:GT}
	Let $f \colon X \to Y$ be a smooth algebraic fiber space whose general fiber $F$
	satisfies $\kappa(F) \geq 0$. Then $Y$ is of general type if and only if 
	$\kappa(X) = \kappa(F) + \dim Y$.
\end{intro-theorem}

By the Easy Addition formula \cite[Cor.~1.7]{Mori}, we always have the elementary
inequality $\kappa(X) \leq \kappa(F) + \dim Y$.  Also recall that Iitaka's
conjecture, to the effect that $\kappa(X) \geq \kappa(F) + \kappa(Y)$, is known when
$Y$ is of general type \cite[Cor.~IV]{Viehweg1}. One half of the theorem is therefore
already known. Our contribution is to prove the other half: the identity
$\kappa(X) = \kappa(F) + \dim Y$ implies that $Y$ is of general type.

\begin{note}
	Using the same techniques, one can show that if $f \colon X \to Y$ is an algebraic
	fiber space with $Y$ not uniruled, $\kappa(F) \geq 0$, and $\kappa(X) = \kappa(F)
	+ \dim Y$, and if $V \subseteq Y$ denotes the smooth locus of $f$, then $V$ is of
	log general type.  This answers positively  another question suggested by \cite{Popa}, when $Y$ is not
	uniruled, but is weaker than \cite[Conj.~3.1]{Popa}. See also \remarkref{rmk:open}.
\end{note}

Our next result deals with smooth algebraic fiber spaces with canonically polarized fibers.

\begin{intro-theorem}\label{thm:fibersGT}
	Let $f \colon X \to Y$ be a smooth algebraic fiber space such that each fiber $F$
	is canonically polarized. If $\kappa (Y) \ge 0$, or if $Y$ is uniruled, then we have
	$\kappa (X) = \kappa (F) + \kappa (Y)$.
\end{intro-theorem}

For technical reasons, the case where $\kappa(Y) = -\infty$ and $Y$ is not uniruled
is missing; note however that a well-known conjecture in birational geometry, called
the \emph{non-vanishing conjecture}, predicts that $\kappa (Y) = -\infty$ is in fact equivalent to $Y$ being uniruled.
Thus up to this prediction, our theorem implies the additivity of the Kodaira dimension for any smooth algebraic fiber space with fibers of this type. Unlike the other results in this paper, 
besides hyperbolicity-type techniques, \theoremref{thm:fibersGT} also relies on the existence and properties of moduli spaces of varieties of general type. We also observe that essentially the same proof works even when $X$ and $Y$ are only assumed to be quasi-projective, with the  base of log Kodaira dimension $\kappa (Y) \ge 0$; see Theorem \ref{thm:fibersGT-quasiproj}.\footnote{Note that since this paper was written, Campana \cite{Campana} has extended this result to morphisms whose fibers have semiample canonical bundle, with no conditions on the base.}

We also prove a more general result that depends on 
the following conjecture of Campana-Peternell \cite[Conj.~2.4]{CP}, which can be seen
as a generalization of the non-vanishing conjecture:

\begin{conjecture*}[Campana-Peternell] 
	Let $Y$ be a smooth projective variety, and assume that $K_Y \sim_{\QQ} A + B$,
	where $A$ and $B$ are $\QQ$-divisors such that $A$ is effective and $B$ is
	pseudo-effective. Then $\kappa (Y) \ge \kappa (A)$.  
\end{conjecture*}

The non-vanishing conjecture is the case $A = 0$. Concretely, it says
that if $K_Y$ is pseudo-effective -- which is equivalent to $Y$ not being uniruled
\cite{BDPP} -- then $\kappa (Y) \ge 0$. It can be easily checked \cite[p.~12--13]{CP}
that the Campana-Peternell conjecture is implied by the existence of good minimal models; 
it therefore holds if $\dim Y \le
3$, since this conjecture in the Minimal Model Program (MMP)
is known to hold in this range. On the other hand, the Campana-Peternell
conjecture is definitely weaker than the full abundance conjecture; in fact, we
believe that it is basically equivalent to the non-vanishing conjecture
\cite{Schnell2}.

\begin{intro-theorem}\label{thm:CP-consequence}
Let $f \colon X \to Y$ be a smooth algebraic fiber space, and assume that  the
Campana-Peternell conjecture holds.  Then $ \kappa (F) + \kappa (Y) \ge \kappa (X)$.
\end{intro-theorem}

\begin{note}
More precisely, we need the Campana-Peternell conjecture for the fibers of the Iitaka
fibration of $Y$ if $\kappa (Y) \ge 0$; or the non-vanishing conjecture for $Y$
itself if $\kappa (Y) = - \infty$.
\end{note}

Iitaka's conjecture predicts of course the opposite inequality $\kappa (X) \ge \kappa
(F) + \kappa (Y)$, for an arbitrary morphism; it is known to be implied by the
existence of a good minimal model for the geometric generic fiber of $f$ by a result
of Kawamata \cite{Kawamata}.  Therefore we obtain the following loosely formulated consequence.

\begin{intro-corollary}
The conjectures of the MMP imply \cite[Conj.~2.5]{Popa}, stating that if $f \colon X \to Y$ is a smooth algebraic fiber space, then $\kappa (X) = \kappa (F) + \kappa (Y)$.
\end{intro-corollary}

In any event, we obtain unconditional results -- and hence the solution to
\cite[Conj.~2.1]{Popa} -- for algebraic fiber spaces over surfaces and threefolds.

\begin{intro-corollary}\label{cor:surfaces}
If $f \colon X \to S$ is a smooth algebraic fiber space, with $S$ a smooth projective
surface, then $\kappa (X) = \kappa (F) + \kappa (S)$.
\end{intro-corollary}

\begin{intro-corollary}
If $f \colon X \to Y$ is a smooth algebraic fiber space, with $Y$ a smooth projective
threefold, then $\kappa (F) + \kappa (Y) \ge \kappa (X)$.
\end{intro-corollary}

In \corollaryref{cor:surfaces} we have equality due to the fact that the Iitaka conjecture is known over surfaces; see \cite{Cao}, and the references therein. This is not yet known when $Y$ is an arbitrary threefold.

\theoremref{thm:CP-consequence} can be easily reduced to the two cases $\kappa (Y) =
- \infty$ and $\kappa (Y) = 0$, and is therefore a consequence of the following two
results. 

\begin{intro-proposition}\label{kappa-negative}
Let $f\colon X \to Y$ be a smooth morphism between smooth projective varieties, with
$Y$ uniruled. Then $\kappa (X) = -\infty$.
\end{intro-proposition}

Thus the additivity of Kodaira dimension for smooth morphisms over varieties $Y$ with
$\kappa (Y) = -\infty $ holds if one assumes the non-vanishing conjecture.  The proof of
\propositionref{kappa-negative} is the first step in the proof of
\theoremref{thm:fibersGT}. It works by reduction to the case $Y = \PP^1$, where a stronger
version is a deep result of Viehweg-Zuo \cite{VZ1}. Therefore the main new input is
the following:

\begin{intro-theorem}\label{thm:kappa=0}
	Let $f\colon X \to Y$ be a smooth algebraic fiber space with $\kappa (Y) = 0$. 
\begin{enumerate}
\item If the Campana-Peternell conjecture holds on $Y$, so for instance if $Y$ has a good minimal model, then 
$\kappa (F) \ge \kappa (X)$.
\item If $Y$ is actually a good minimal model, in other words if $K_Y \sim_{\QQ} 0$, then $\kappa (X) = \kappa (F)$. 
Moreover, we have $P_m (F) \ge P_m (X)$ for all $m \ge 1$, with equality if $Y$ is simply connected (e.g.
Calabi-Yau).
\end{enumerate}
\end{intro-theorem}

Thus \cite[Conj.~2.5]{Popa} holds unconditionally when $Y$ is a good minimal model
with $\kappa (Y) = 0$. When $Y$ is an abelian variety, this is a special case of
\cite[Thm.~A(2)]{MP}. Note for completeness that the conjecture also holds when 
$X$ is a good minimal model with $\kappa (X) = 0$, since then so is $Y$ by \cite{TZ}.

The proofs of the results in this paper rely on a Hodge module construction from
\cite{PS1, PS2}, which was in turn heavily inspired by techniques of Viehweg-Zuo in their work on hyperbolicity. Because of the smoothness assumption, 
semistable reduction and the technical passage to logarithmic Higgs bundles are
however not needed. The proof of \theoremref{thm:kappa=0} is somewhat more delicate
(than for instance that of \theoremref{thm:GT}), needing both an additional technical
ingredient and important analytic results from \cite{PT,CaP} (see also \cite{HPS}).

\begin{remark*}
	The disadvantage of this method is the lack of symmetry between the
	assumptions and the conclusion. Say $f \colon X \to Y$ is a smooth algebraic fiber
	space, and $D$ a divisor on $Y$. Roughly speaking, we start from the assumption
	that $m K_X - \fu D$ is \emph{effective} for some $m \geq 1$; and the
	conclusion is that $m' K_Y - D$ is \emph{pseudo-effective} for $m' \gg m$. The
	Campana-Peternell conjecture removes this asymmetry, and allows us to put 
	``pseudo-effective'' or ``effective'' in both places; but it would of course be
	desirable to have unconditional results.
\end{remark*}

We conclude by observing that the Campana-Peternell conjecture, together with the techniques above, also leads to the following characterization of $\kappa(Y)$ that can be seen as a generalization of \theoremref{thm:GT}.

\begin{intro-corollary}\label{cor:CP-consequence}
	Let $f \colon X \to Y$ be a smooth algebraic fiber space, with general fiber $F$ satisfying $\kappa (F) \ge 0$.
	If the Campana-Peternell conjecture holds on $X$ and $Y$, then
	\[
		\kappa(Y) = \max \menge{\kappa(L)}{\text{$mK_X - \fu L$ is pseudo-effective for
		some $m \geq 1$}},
	\]
	where the maximum is taken over all line bundles on $Y$.
\end{intro-corollary}

\subsection*{Acknowledgements} 
We thank Dori Bejleri, Gavril Farkas, Fanjun Meng, Sung Gi Park, Valentino Tosatti and Ziquan Zhuang for their
comments and suggestions. In particular, Zhuang's help was crucial for the proof of
\theoremref{thm:fibersGT}. M.P.\ is partially supported by NSF grant DMS-2040378.
Ch.S.\ is partially supported by NSF grant DMS-1551677 and by a Simons Fellowship
from the Simons Foundation. He thanks the Max-Planck-Institute for providing him with
excellent working conditions during his stay in Bonn.

\section{Proofs}

\subsection{Proof of Theorem \ref*{thm:GT}}

We start by considering a general construction.
Let $A$ be an ample line bundle on $Y$. According to \cite[Prop.~1.14]{Mori}, the
condition that $\kappa(X) = \kappa(F) + \dim Y$ is equivalent to having
\[
	H^0 (X, \omega_X^{\otimes m } \otimes f^* A^{-1}) \neq 0
\]
for some $m \geq 1$. Using the projection formula, this is equivalent to
the existence of a nontrivial (hence injective) morphism
\begin{equation}\label{eqn:inclusion}
	L : = A \otimes (\omega_Y^{-1})^{\otimes m}  \into f_* \omega_{X/Y}^{\otimes m}.
\end{equation}
Note that by the invariance of plurigenera, the sheaf $f_* \omega_{X/Y}^{\otimes m}$
is locally free of rank $P_m(F)$ on all of $Y$.

We now use Viehweg's well-known fiber product trick, in the following form: let
$X^{(m)}$ denote the $m$-fold fiber
product $X \times_{Y}  \cdots \times_{Y} X$, with its induced morphism
$f^{(m)}: X^{(m)} \rightarrow Y$. Note that this is still a smooth algebraic fiber
space, with general fiber isomorphic to the $m$-fold product $F \times \cdots \times F$. Moreover, 
\begin{equation}\label{eqn:Vie}
f^{(m)}_* \omega_{X^{(m)}/Y}^{\otimes m} \simeq (f_* \omega_{X/Y}^{\otimes m})^{\otimes m}
\end{equation}
for example by \cite[Cor.~4.11]{Mori}. This gives us an inclusion
\[
	L^{\otimes m} \into  f^{(m)}_* \omega_{X^{(m)}/Y}^{\otimes m}.
\]
By replacing $X$ by $X^{(m)}$, and the fiber space $f$ by $f^{(m)}$, we are therefore
allowed to assume from the beginning that we are working with a fiber space $f$ such that
\begin{equation}\label{eqn:scn}
	H^0 (X, B^{\otimes m}) \neq 0  \quad \text{where \,\, $B = \omega_{X/Y} \otimes f^*
	L^{-1}$.}
\end{equation}
Under this assumption, one can bring into play the machinery used in \cite{PS1,
PS2}, strongly inspired in turn by the celebrated method developed by Viehweg-Zuo
\cite{VZ1,VZ2}, based on the existence of so-called Viehweg-Zuo sheaves. Note however that, unlike in the 
works above, the line bundle $L$ we consider here is usually not positive any more. 

First, using \cite[Thm.~2.2]{PS2} and denoting $\shA_Y = \Sym T_Y$, ($\ref{eqn:scn}$) implies that there exists a  graded $\shA_Y$-module $\shG_{\bullet}$ that is coherent over $\shA_Y$, and has the following among its special properties:
\begin{enumerate}
\item 
One has $\shG_0 \simeq L$.
\item \label{en:VZc}
Each $\shG_k$ is torsion-free.
\item
There exists a regular holonomic $\Dmod_Y$-module $\Mmod$ with good filtration
$F_{\bullet} \Mmod$, and an inclusion of graded $\shA_Y$-modules $\shG_{\bullet} \subseteq
\gr_{\bullet}^F \! \Mmod$.
\item 
The filtered $\Dmod$-module $(\Mmod, F_{\bullet} \Mmod)$ underlies a polarizable Hodge module $M$ on $Y$, with strict support $Y$, and $F_k \Mmod = 0$ for $k < 0$.
\end{enumerate}
We next follow closely \cite[\S3.2]{PS2}. The argument is based on considering the chain of 
$\shO_Y$-module homomorphisms
$$0 \longrightarrow \shG_0 \overset{\theta_0}{\longrightarrow}  \shG_{1}  \otimes  \OmY^1 
\overset{\theta_{1} \circ \id}{\longrightarrow}  \shG_{2}  \otimes  
(\OmY^1)^{\otimes 2} \longrightarrow \cdots$$
obtained from the graded $\shA_Y$-module structure on $\shG_\bullet$ and the generalized 
Kodaira-Spencer maps 
$$\theta_k:  \shG_k  \longrightarrow \shG_{k+1}  \otimes \OmY^1,$$
and originated in work of Kov\'acs and Viehweg-Zuo. Using the identification $\shG_0 \simeq L$ and 
the fact that ${\rm Ker}(\theta_k)^\vee$ is a weakly positive sheaf for each $k$,\footnote{This requires the
negativity of the kernels of generalized Kodaira-Spencer maps for pure Hodge modules
with strict support $Y$, proved in \cite{PW} using results by Zuo and Brunebarbe; see the references in
\emph{loc. cit.}}
the exact same argument as in the proof of
\cite[Thm.~3.5]{PS2} shows that at least one of the following two cases must
happen:

\begin{enumerate}
	\item The line bundle $L^{-1} = A^{-1} \tensor \omY^{\otimes m}$ is pseudo-effective.
	\item There exists a weakly positive sheaf $\mathcal{W}$ on $Y$ and an inclusion
$$\mathcal{W} \otimes L \into (\Omega_Y^1)^{\otimes N}$$
for some positive integer $N> 0$.
\end{enumerate}

In the first case, we immediately get that $\omY$ is big, and hence that $Y$ is of
general type. In the second case, we argue as follows. We rewrite the above inclusion
as
\begin{equation} \label{eq:inclusion}
	\mathcal{W} \otimes A \into (\Omega_Y^1)^{\otimes N} \otimes \omega_Y^{\otimes m},
\end{equation}
and observe that $\mathcal{W} \tensor A$ is big in the sense of Viehweg
\cite[Lem.~3.2]{PS2}. Now $\det \Omega_Y^1 = \omega_Y$, hence there exists also a (split) inclusion 
$$\omega_Y^{\otimes m}\hookrightarrow (\Omega_Y^1)^{\otimes m\cdot \dim Y}.$$
Putting everything together we deduce the existence of an inclusion
$$\shH \hookrightarrow  (\Omega_Y^1)^{\otimes M} $$ 
for some $M>0$, where $\shH$ is a big torsion-free sheaf. The result of Campana-P\u
aun \cite[Thm.~7.11]{CPa} implies then that $Y$ is of general type. (Campana and
P\u{a}un only state the result with $\shH$ being a line bundle, but their proof
still works when $\shH$ has higher rank; see also \cite{Schnell:CP} for a streamlined
proof of this important result.)

\begin{remark}[{\bf The non-smooth case}]\label{rmk:open}
As mentioned in the introduction, when $Y$ is not uniruled (which we recall implies that $\omega_Y$ is pseudo-effective by \cite{BDPP})  the proof can be extended without much effort to the general case, when $f$ is only smooth over an open set $V \subseteq Y$, to deduce that $V$ is of log general type. By blowing up it is immediate to reduce to the case 
where $Y \smallsetminus V = D$, a simple normal crossing divisor. The only difference in the argument is that 
($\ref{eqn:Vie}$) is now an isomorphism only away from $D$, which in practice means that in the proof above we have to
replace $L$ by $L (- rD)$ for some integer $r > 0$. The same steps now lead to the bigness of a line bundle of the form
$$\omega_Y^{\otimes a} \otimes \shO_Y(b D)$$ 
for some $a, b> 0$, and the pseudo-effectivity of $\omega_Y$ is then needed to deduce the bigness of $\omega_Y (D)$ 
by multiplying by a suitable multiple of either $\omega_Y$ or $\shO_Y (D)$. (Compare
also with the proof of \cite[Thm.~4.1]{PS2}.)

Note also that if $\kappa (Y) \ge 0$ and the complement of $V$ in $Y$ has codimension
at least $2$, it is not hard to check that $\kappa (V) = \kappa (Y)$; see e.g.
\cite[Lem.~2.6]{MP}. Thus under this assumption the conclusion is still that $Y$ 
is of general type.
\end{remark}

\subsection{Proof of Theorem \ref*{thm:fibersGT}}
We divide the proof into a few steps. 

\noindent
\emph{Step 1.}
In this step we deal with the case when $Y$ is uniruled,\footnote{As mentioned in the
introduction, conjecturally this is equivalent to $\kappa (Y) = - \infty$.} so in
particular $\kappa (Y) = -\infty$, by easy reduction to the case $Y = \PP^1$, in
which case a theorem of Viehweg-Zuo applies. We are in fact proving
\propositionref{kappa-negative}; no assumption on the fibers of  $f$ is necessary.
See also \remarkref{rmk:psef} below for an alternative proof based on the techniques
used in \theoremref{thm:GT}.

The hypothesis means that there exists a variety $Z$, which can be assumed to be smooth and projective, and a dominant rational map 
$$
Z \times \PP^1  \cdots \to Y.
$$
By resolving the indeterminacies of this map, we obtain a generically finite surjective morphism $\varphi \colon W \to Y$, such that $W$ admits a morphism $g \colon W \to Z$ with general fiber $\PP^1$. We consider the diagram
\[
\begin{tikzcd}
	\widetilde{X} \rar{\psi} \dar{\tilde{f}} \arrow[bend right=40,swap]{dd}{h} & X \dar{f} \\
W \rar{\varphi} \dar{g} & Y \\
Z &
\end{tikzcd}
\]
where $h = g \circ \tilde{f}$. Since $\psi$ is generically finite we have $\kappa (\widetilde{X}) \ge \kappa (X)$ (see 
\cite[Cor.~2.3(i)]{Mori}), and therefore it suffices to
show that $\kappa (\widetilde{X}) = - \infty$.
Note now that the general fiber $H$ of $h$ admits a smooth morphism (with fiber $F$) to the general fiber of $g$, i.e. 
to $\PP^1$. By \cite[Thm.~2]{VZ1} it follows that $\kappa (H) = - \infty$, and therefore by Easy Addition  (see 
\cite[Cor.~1.7]{Mori}) we also have $\kappa (\widetilde{X}) = - \infty$. As promised, this proves
\propositionref{kappa-negative}.

\noindent
\emph{Step 2.}
In this step we assume that $\kappa (Y) \ge 0$, and show that we can reduce to the case $\kappa (Y) = 0$. 
First, since the fibers of $f$ are of general type, the subadditivity 
$$\kappa (X) \ge \kappa (F) + \kappa (Y)$$
conjectured by Iitaka holds by \cite{Kollar}. To prove \theoremref{thm:fibersGT}, it therefore suffices to show 
\begin{equation}\label{eqn:super}
\kappa (F) + \kappa (Y) \ge \kappa (X).
\end{equation}
We show the reduction for this inequality.

To this end, we consider the Iitaka fibration $g \colon Y \to Z$, which after base change can be assumed to be a morphism, with $Z$ smooth and projective. We again denote $h = g \circ f$, with general fiber $H$, and we also denote the general fiber of $g$ by $G$. Thus we have a smooth morphism $H \to G$, with fiber $F$, and since $\kappa (G) = 0$, we may assume that $\kappa (F) \ge \kappa (H)$. On the other hand, by definition we have $\dim Z = \kappa (Y)$, and 
Easy Addition applied to $h$ gives 
$$\kappa (H) + \dim Z \ge \kappa (X).$$
Putting the two inequalities together, we obtain ($\ref{eqn:super}$).
 
\noindent
\emph{Step 3.}
Under the assumption that the fibers of $f$ are canonically polarized, by Taji \cite[Thm. 1.5]{Taji} we also know the validity of Campana's isotriviality conjecture, and hence the Kebekus-Kov\'acs conjecture, stating that if $\kappa (Y) \ge 0$, then $\kappa (Y) \ge \Var(f)$. Using this in combination with Step 2, we may therefore assume that $\Var(f) = 0$, i.e. that $f$ is birationally isotrivial.\footnote{It is worth noting that in this proof, the hyperbolicity-type techniques present in the other results in this paper are hidden in the results from \cite{VZ1} and \cite{Taji}.}

\noindent
\emph{Step 4.}
In this final step, we show the assertion in the theorem under the assumption that $\Var(f) = 0$ (with no assumptions on $Y$). This is a consequence of the existence and properties of the moduli space of canonically polarized varieties, as suggested to us by Ziquan Zhuang.

Concretely, the family $f$ induces a  morphism 
$$\varphi \colon Y \to {\bf M}$$
to the appropriate coarse moduli space of canonically polarized varieties  \cite{Viehweg2,Kollar2}. Since $f$ is birationally isotrivial, 
it follows that $\varphi$ is constant on an open set, and therefore constant everywhere; in other words the family $f$ is isotrivial. 
Since varieties of general type have finite automorphism group (${\bf M}$ is the coarse moduli space of a Deligne-Mumford stack), it is therefore well-known that there is a finite \'etale base change $\widetilde{Y} \to Y$ such that 
$$\widetilde{X} : = X \times_Y \widetilde{Y}  \simeq F \times \widetilde{Y},$$
where $F$ is any fiber of $f$. We thus have 
$$\kappa (X) = \kappa (\widetilde{X}) = \kappa (F) + \kappa (\widetilde{Y}) =\kappa (F) + \kappa (Y).$$
This concludes the proof.

\begin{remark}\label{rmk:psef}
\propositionref{kappa-negative} can be rephrased as saying that if $f\colon X \to Y$
is smooth and $\kappa (X) \ge 0$, then $K_Y$ is pseudo-effective. Here is an
alternative proof of this fact, resembling that of \theoremref{thm:GT}. We consider the construction in the proof of that theorem, only now we take $A = \shO_Y$, which is possible because $\kappa (X) \ge 0$. The exact same proof shows then that either we have directly 
that $K_Y$ is pseudo-effective (case (i)), or that there is an inclusion 
$$\mathcal{W} \into (\Omega_Y^1)^{\otimes N} \otimes \omega_Y^{\otimes m},$$
where $\mathcal{W}$ has pseudo-effective determinant (case (ii)).  But \cite[Thm.~7.6]{CPa} says precisely that in the latter case $K_Y$ is again pseudo-effective.
\end{remark}

\begin{remark}
	In fact, recent work in symplectic geometry can be used to prove the following
	strengthening of \propositionref{kappa-negative}: \emph{Let $f \colon X \to Y$ be a
	smooth algebraic fiber space. If $K_X$ is pseudo-effective, then $K_Y$ is also
	pseudo-effective.} Here is a sketch of the proof. If $K_Y$ is not pseudo-effective,
	then $Y$ is covered by rational curves. After restricting to a smooth rational
	curve whose normal bundle has nonnegative degree, we obtain a smooth algebraic
	fiber space $g \colon Z \to \PP^1$ such that $K_Z$ is pseudo-effective. But a
	very recent theorem by Pieloch \cite[Thm.~1.1]{Pieloch} -- answering a question by 
	Starr \cite{Jason} -- says that $Z$ is covered by
	sections of $g$, and therefore uniruled. This is a contradiction.
\end{remark}

\begin{remark}[{\bf The quasi-projective case}]
A very similar proof to that of Theorem \ref{thm:fibersGT} leads in fact to a statement for smooth morphisms between quasi-projective varieties, addressing an analogous case of the more general \cite[Conj. 3.1]{Popa}. We thank Sung Gi Park for pointing this out. 
\end{remark}

\begin{theorem}\label{thm:fibersGT-quasiproj}
	Let $f \colon U \to V$ be a smooth projective algebraic fiber space with canonically polarized fibers. If $\kappa (V) \ge 0$, then we have $\kappa (U) = \kappa (F) + \kappa (V)$.
\end{theorem}

\begin{proof}
This is the analogue of Steps 2.-4. in the proof of Theorem \ref{thm:fibersGT}.

Since the fibers are of general type, the analogue of Iitaka's conjecture, namely 
$$\kappa (U) \ge \kappa (F) + \kappa (V)$$
holds by \cite[Thm. 1.3]{KP}. Therefore we only need to focus on the opposite inequality.
We can then reduce to the case $\kappa (V) = 0$ by using the analogue of the Iitaka fibration and the Easy Addition theorem
in the case of open varieties; see \cite[Thm. 4 and 5]{Iitaka}.

Now again by \cite[Thm. 1.5]{Taji}, which is also valid in the quasi-projective case,
we may therefore assume that ${\rm Var}(f) = 0$. Then the exact same argument as in Step 4 of the proof of Theorem \ref{thm:fibersGT}, based on the properties of the moduli space of varieties of general type, shows that the relative canonical model $U'$ of $U$ over $V$ has an \'etale cover 
$\widetilde{U}$ which is birational to $F \times V$. The statement then follows, since on one hand $\kappa (U) = \kappa (U')$ as $U$ and $U'$ admit proper birational morphisms from a fixed smooth quasi-projective variety $W$ (see \cite[\S3]{Iitaka}), and on the other hand the log Kodaira dimension is preserved under \'etale covers (see \cite[Thm. 3]{Iitaka}).
\end{proof}

Note however that a more general statement, removing the condition on the base, and allowing for fibers with semiample canonical bundle, appeared very recently in \cite{Campana}.

\begin{remark}
If confirmed, the results of \cite{WW} would lead to validity of the statement of Theorem \ref{thm:fibersGT-quasiproj} in the case when the fibers of $f$ are only assumed to be of general type. One would only need to replace $f$ by its relative canonical model, and apply the same argument.
\end{remark}

\subsection{Proof of Theorem \ref*{thm:CP-consequence}}
If $\kappa (Y) = - \infty$ the result follows from the non-vanishing conjecture together with \propositionref{kappa-negative}, while  if $\kappa (Y) = 0$, it follows from \theoremref{thm:kappa=0}(i). If $\kappa (Y) > 0$, we can reduce to the case 
$\kappa = 0$ using the standard argument involving the Iitaka fibration: we may assume that this is a morphism
$g \colon Y \to Z$, with general fiber $G$ satisfying $\kappa (G) = 0$. Therefore the general fiber $H$ of $ h = g \circ f$ has 
a smooth morphism $H \to G$ with fiber $F$, and consequently $\kappa (F) \ge \kappa (H)$. On the other hand, Easy
Addition for $h$ gives 
$$\kappa (H)+ \kappa (Y) = \kappa (H) + \dim Z \ge \kappa (X).$$
Hence the remaining point is to prove \theoremref{thm:kappa=0}.

\subsection{Proof of Theorem \ref*{thm:kappa=0}}
If $\kappa (F) = - \infty$, then we also have $\kappa (X) = - \infty$ by Easy Addition. We may therefore assume throughout 
that $\kappa (F) \ge 0$. The key players are again the vector bundles 
$$\shF_m : = f_* \omega_{X/Y}^{\otimes m}, \,\,\,\,\,\,{\rm for} \,\,\,\,m \ge 1,$$
We also consider their twists
$$\shE_m : = \shF_m \otimes \omega_Y^{\otimes m} \simeq f_* \omega_X^{\otimes m} \,\,\,\,\,\,{\rm for} \,\,\,\,m \ge 1,$$
for which we have 
$$r_m := \rk(\shE_m)  =  \rk(\shF_m) = P_m (F) \,\,\,\,\,\,{\rm and} \,\,\,\,\,\, h^0 (Y, \shE_m) = P_m (X).$$
We have for each $m$ a (split) inclusion
$$\det \shF_m \hookrightarrow \shF_m^{\otimes r_m},$$
and therefore using the Viehweg fiber product trick just as in the proof of \theoremref{thm:GT}, we obtain an inclusion
$$(\det \shF_m)^{\otimes m} \hookrightarrow f^{(m r_m)}_* \omega_{X^{(mr_m)}/ Y}^{\otimes m}.$$
Proceeding precisely as in the proof of that theorem (where the role of $L$ there is now played by $\det \shF_m$), the arguments from \cite{PS2} lead then to one of the following two possibilities:

\smallskip

\noindent
(1) $(\det \shF_m)^{-1}$ is pseudo-effective.

\noindent
(2) There exists a weakly positive sheaf $\mathcal{W}$ on $Y$,   and an inclusion
$$\mathcal{W} \otimes \det \shF_m \hookrightarrow (\Omega_Y^1)^{\otimes N}$$
for some positive integer $N> 0$.

Recall on the other hand that $\shF_m$ is a weakly positive sheaf by \cite[Thm.~III]{Viehweg1}, 
and therefore $\det \shF_m$ is a pseudo-effective line bundle, e.g. by \cite[Lem.~1.6 (6)]{Viehweg1}. 
If (1) holds, then both $\det \shF_m$ and $(\det \shF_m)^{-1}$ are
pseudo-effective, and so $\det \shF_m \equiv 0$ (i.e. it is numerically trivial).

On the other hand, if (2) holds, by taking the saturation of the left hand side inside $(\Omega_Y^1)^{\otimes N}$, we obtain an  exact sequence 
$$0\longrightarrow \mathcal{W}' \otimes \det \shF_m \longrightarrow (\Omega_Y^1)^{\otimes N} \longrightarrow \shQ \longrightarrow 0$$
where $\mathcal{W}'$ is a weakly positive sheaf, and $\shQ$ is  torsion-free. We also know that 
$\det \shQ$ is pseudo-effective; this follows from
\cite[Thm.~1.3]{CPa}, which says that if $K_Y$ is pseudo-effective, then any quotient
of any tensor power of $\Omega_Y^1$ has pseudo-effective determinant.
Passing to determinants, we then obtain an isomorphism of the form 
\begin{equation}\label{decomposition}
\omega_Y^{\otimes a} \simeq (\det \shF_m)^{\otimes b} \otimes P,
\end{equation}
where $a, b>0$ and $P$ is a pseudo-effective line bundle.

\begin{remark}
Up to here we only used that $\kappa (F) \ge 0$ and $K_Y$ is pseudo-effective. 
\end{remark}

If we assume that $K_Y \sim_{\QQ} 0$, from ($\ref{decomposition}$) we again deduce that
$(\det \shF_m)^{-1}$ is pseudo-effective, and therefore as above $\det \shF_m \equiv 0$.
It follows from \cite[Thm.~5.2]{CaP} (see also \cite[Cor.~27.2]{HPS}) that $\shF_m$ is a hermitian flat vector bundle,
and therefore comes from a unitary representation of the fundamental group. If we
assume in addition that $Y$ is simply connected, we then obtain 
$\shF_m \simeq \shO_Y^{\oplus r_m}$, or equivalently
$$\shE_m \simeq  {(\omega_Y^{\otimes m})}^{\oplus r_m}.$$
Note that by hypothesis $\omega_Y^{\otimes m} \simeq \shO_X$ when $m$ is sufficiently divisible, and otherwise vanishes. 
Consequently $P_m (X) = P_m (F)$ when $P_m (Y) = 1$, and $P_m (X) = 0$ otherwise. 
In general we only have $r_m \ge h^0 (Y, \shE_m)$ (because unitary representations are
semisimple), and therefore by the same argument we obtain 
\[
	P_m (F) \ge P_m (X) \quad \text{for all $m \ge 1$},
\]
hence also $\kappa (F) \ge \kappa (X)$. On the other hand, a result of Cao-P\u aun
\cite[Thm.~5.6]{CaP} states, for any algebraic fiber space $f$, that if $c_1 (\det
\shF_m) = 0 \in H^2 (X, \RR)$ for some $m \ge 2$, then Iitaka's conjecture holds for
$f$. Cao and P\u{a}un prove this result under the assumption that $c_1 (\det \shF_m) = 0
\in H^2 (X, \ZZ)$, but because of \cite[Thm.~27.2]{HPS}, the weaker assumption is
sufficient to conclude that $\shF_m$ is a hermitian flat bundle; the rest of the
argument then proceeds as in \cite{CaP}. In our case, this gives the opposite inequality
$\kappa (X) \ge \kappa (F)$, and concludes the proof of (ii).

We are thus left with proving (i).  We need the following lemma.

\begin{lemma}
	Let $E$ be a nonzero vector bundle on a projective variety $X$. If $h^0(X, E)
	\geq \rk E + 1$, then there is a line bundle $L$ with $h^0 (X, L) \ge 2$ and an
	inclusion $L \into E^{\tensor r}$ for some $1 \leq r \leq \rk E$.
\end{lemma}

\begin{proof}
	Let $\shF \subseteq E $ be the torsion-free subsheaf generated by the global sections of
	$E$, and let $1 \leq r \leq \rk E$ be the generic rank of $\shF$. Choose $r$
	global sections $s_1, \dotsc, s_r \in H^0(X, E)$ that generate the stalk of $\shF$ at a
	general point of $X$. The resulting morphism $\OX^{\oplus r} \to
	\shF$ is injective; it cannot be surjective, for otherwise it would be an
	isomorphism, and then $h^0(X, E) = h^0(X, \shF) = r \leq \rk E$, contradicting
	the assumption that $h^0(X, E) \geq \rk E + 1$. Let $x \in X$ be one of the points
	where the morphism fails to be surjective on stalks. Since $\shF$ is globally
	generated, depending on the situation we can do one of the following two things:
	
	$\bullet$~If $s_1, \ldots, s_r$ remain linearly independent at $x$, 
	we can find an additional global section $s \in H^0(X, E)$ such that
	such that $s(x)$ does not lie in the linear span of $s_1(x), \dotsc, s_r(x)$
	inside the stalk $E_x$. It follows that, among the $r+1$ sections 
	\[
		s_1 \wedge \dotsb \wedge s_r, \quad 
		s_1 \wedge \dotsb \wedge s_{i-1} \wedge s \wedge s_{i+1} \wedge \dotsb \wedge
		s_r \quad \text{for $i=1, \dotsc, r$}
	\]
	of the bundle $\bigwedge^r E$, at least $2$ are linearly indepedent. 
		
	$\bullet$~If $s_1, \ldots, s_r$ do not remain linearly independent at $x$, we can in any case find other global 
	sections $t_1, \ldots, t_r \in H^0 (X, E)$ such that $t_1(x), \ldots, t_r (x) \in E_x$ form part of a minimal system 
	of generators of $\shF_x$. Then 
	\[   
	       s_1 \wedge \dotsb \wedge s_r \quad {\rm and} \quad t_1 \wedge \dotsb \wedge t_r
	\]
	are linearly independent global sections of the bundle $\bigwedge^r E$.

	Either way, by construction all these sections come from global sections of the sheaf $\bigwedge^r \shF$. Let
	$L$ be the reflexive hull of $\bigwedge^r \shF$; then $L$ is a line bundle with
	$h^0(X, L) \geq 2$. We also have inclusions
	\[
		L \into \bigwedge^r E \into E^{\tensor r}
	\]
	because $E$ is locally free.
\end{proof}

Let us now suppose for the sake of contradiction that $P_m(X) \geq P_m(F) + 1$. The
lemma, applied to the vector bundle $\shE_m$, then gives us an inclusion 
$$L \into \shE_m^{\tensor r} \simeq \shF_m^{\otimes r} \otimes \omega_Y^{\otimes mr}$$ 
for some integer $1 \leq r \leq r_m$, where $L$ is a line bundle
with $h^0(X, L) \geq 2$, and so $\kappa(L) \geq 1$. Arguing as above, using the fiber product trick, we obtain one of the 
following two possibilities: 

\smallskip

\noindent
(1) $L^{-1} \otimes \omega_Y^{\otimes mr}$ is pseudo-effective.

\noindent
(2) There exists a weakly positive sheaf $\mathcal{W}$ on $Y$,   and a short exact sequence 
$$0 \to \mathcal{W} \otimes L   \to (\Omega_Y^1)^{\otimes N} \otimes \omega_Y^{\otimes mr} \to \shQ \to 0$$
for some positive integer $N> 0$.

In case (1), since $\kappa(L) \geq 1$, the Campana-Peternell conjecture implies that $\kappa(Y) \geq 1$, contradicting our assumption that $\kappa(Y) = 0$. 

In case (2), as above $\det \shQ$ must be a pseudo-effective line bundle by \cite[Thm.~1.3]{CPa}. 
Passing to
determinants in the short exact sequence in (2),  we now obtain an isomorphism
\[
	\omega_Y^{\tensor a} \simeq L^{\tensor b} \tensor P
\]
where $a,b > 0$ and $P$ is a pseudo-effective line bundle. We then conclude just as
in case (1).

\subsection{Proof of Corollary \ref*{cor:CP-consequence}}
Let $f \colon X \to Y$ be a smooth algebraic fiber space whose general fiber $F$
satisfies $\kappa(F) \geq 0$. We use the notation 
\[
	\kappa_f(Y) := \max \menge{\kappa(Y, L)}%
	{\text{$mK_X - \fu L$ is pseudo-effective for some $m \geq 1$}},
\] 
and show that $\kappa(Y) = \kappa_f(Y)$ provided that $\kappa (F) \ge 0$. One
inequality is easy. Our assumption that $\kappa(F) \geq 0$ implies that the relative
canonical divisor $K_{X/Y}$ is pseudo-effective; the reason is that
$\omega_{X/Y}$ has a canonical singular hermitian metric with semi-positive curvature
\cite[Thm.~4.2.2]{PT} (see also \cite[Thm.~27.1]{HPS}). Since $K_X - \fu K_Y \equiv
K_{X/Y}$, we get $\kappa_f(Y) \geq \kappa(Y)$. 

To prove the other inequality, let $L$ be a line bundle on $Y$ such that $mK_X
- \fu L$ is pseudo-effective, and such that $\kappa_f(Y) = \kappa(Y, L)$. 
We may assume that $\kappa(Y, L)\ge 0$, and consider an integer $n \ge 1$ such that the linear system $\abs{nL}$ gives the Iitaka 
fibration of $L$.  After
replacing $Y$ by a log resolution of the base locus of $\abs{nL}$, and $f \colon X
\to Y$ by the fiber product, we can assume that $nL \sim B + E$ with $B$ semi-ample
and $E$ effective; consequently, $nmK_X - f^* B \sim n(mK_X - f^* L) + E$ is still
pseudo-effective. Moreover, because of the assumption on $n$ we have 
that $\kappa (Y, B) = \kappa (Y, L)$, hence after replacing $L$ by $B$ we may assume from the
beginning that $L$ is semi-ample. The Campana-Peternell conjecture (on $X$) implies
that $mK_X - \fu L$ is effective
for $m \gg 0$; this is proved in \cite{Schnell2}. We can now apply exactly
the same argument as in the proof of \theoremref{thm:GT}, up to the line containing
\eqref{eq:inclusion}, and obtain a short exact sequence
\[
	0 \to \mathcal{W} \tensor L \to (\Omega_Y^1)^{\otimes N} \otimes \omega_Y^{\otimes m}
	\to \mathcal{Q} \to 0
\]
for some $N \geq 1$, with $\mathcal{W}$ weakly positive. (Concretely, $\mathcal{W}$
is a torsion-free sheaf with a semi-positively curved singular hermitian metric.)
Since $m \geq 1$, both $K_Y$ and $\det \mathcal{Q}$ must therefore be pseudo-effective:
this follows by combining
\cite[Thm.~7.6]{CPa} (to conclude that $K_Y$ is pseudo-effective) and
\cite[Thm.~1.3]{CPa} (to conclude that any quotient of any tensor power of
$\Omega_Y^1$ has pseudo-effective determinant). Passing to determinants in the short
exact sequence above, we see that there are two integers $a,b \geq 1$ such that
\[
	a K_Y \sim bL + \det \mathcal{Q} + \det \mathcal{W}.
\]
As both $\det \mathcal{Q}$ and $\det
\mathcal{W}$ are pseudo-effective, the Campana-Peternell conjecture (on $Y$) now
gives us the desired inequality $\kappa(Y) \geq \kappa(Y, L) = \kappa_f(Y)$.

\bibliographystyle{amsalpha}
\bibliography{bibliography}

\end{document}